\documentclass[11pt]{article}
\usepackage{graphicx}
\usepackage{amssymb,amsmath,amsthm}
\usepackage[latin1]{inputenc}
\usepackage[english]{babel}
\def\N{\mathbb{N}}
\def\Pr{\mathbb{P}}

\def\ZZ{\mathbb{Z}}

\def\le{\leqslant}
\def\ge{\geqslant}

\def\->{\rightarrow}
\def\<{\langle}
\def\>{\rangle}

\newtheorem{theorem}{Theorem}[section]
\newtheorem{lemma}[theorem]{Lemma}
\newtheorem{proposition}[theorem]{Proposition}
\newtheorem{corollary}[theorem]{Corollary}
\newtheorem{obs}[theorem]{Observation}

\newtheorem{problem}[theorem]{Problem}

\theoremstyle{remark}
\newtheorem{remark}{Remark}
\theoremstyle{remark}
\theoremstyle{remark}

\newcommand{\ssq}{\subseteq}

\begin{document}

\title{$k$-sums in abelian groups}
\date{ }
\author{
Benjamin Girard
\thanks{IMJ, \'{E}quipe Combinatoire et Optimisation, 
Universit\'{e} Pierre et Marie Curie~(Paris $6$), $4$ place Jussieu, 75005 Paris, France, email: \texttt{bgirard@math.jussieu.fr}.}
\and
Simon Griffiths
\thanks{IMPA, Estrada Dona Castorina $110$, Rio de Janeiro, Brasil, $22460$-$320$, email: \texttt{sgriff@impa.br}. Research supported by CNPq Proc. 500016/2010-2.}
\and
Yahya ould Hamidoune
}

\maketitle

\begin{abstract} 
Given a finite subset $A$ of an abelian group $G$, we study the set $k \wedge A$ of all sums of $k$ distinct elements of $A$. 
In this paper, we prove that $|k \wedge A| \ge |A|$ for all $k\in \{2,\dots ,|A|-2\}$, unless $k\in \{2,|A|-2\}$ and $A$ is a coset of an elementary $2$-subgroup of $G$. 
Furthermore, we characterize those finite sets $A\ssq G$ for which $|k \wedge A|=|A|$ for some $k\in \{2,\dots ,|A|-2\}$. 
This result answers a question of Diderrich.  Our proof relies on an elementary property of proper edge-colourings of the complete graph.
\end{abstract}

\section{Introduction}

Sumsets are the central object of study in Additive Combinatorics.  In particular, the sumset 
\[ A+B\, :=\, \{a+b:\,a\in A, b\in B\}\, \]
and the $k$-fold sumset
\[ kA\, :=\, \{a_1+\dots + a_k: \, a_i\in A\}\, \]
of subsets $A,B$ of an abelian group $G$ have been extensively studied.  For background and references we refer the reader to the course of Ruzsa~\cite{Ruzsa09} and the book of Tao and Vu~\cite{TV}.  See also recent results of Petridis~\cite{P1}.

In this paper, we consider sums of $k$ \em distinct \em elements of $A$, which we will call \em $k$-sums\em.
Given an abelian group $G$ (not necessarily finite), any finite subset $A$ of $G$ will be referred to as an \em additive set\em, and
\[k\wedge A\, :=\, \left\{\sum_{a\in B}a:\, B \ssq A, |B|=k\right\}\]
will denote the set of $k$-sums of $A$.  Many results concerning $k$-sums of subsets of abelian groups focus on the cases $k=2$ and $k=3$.  See for example the survey of Lev~\cite{Lev05} and the papers of Lev~\cite{Lev02}, Hamidoune, Llad\' o and Serra~\cite{HLS00}, and Gallardo, Grekos, Habsieger, Hennecart, Landreau and Plagne~\cite{GGHHLP02}.  
On the other hand, $k$-sums for general $k$ have been studied more extensively for sequences than for sets, with a particular emphasis on determining whether $0$ is a $k$-sum, see the seminal article of Erd\H os, Ginzburg and Ziv~\cite{EGZ61}, the article of Bollob\' as and Leader~\cite{BL99}, the survey of Gao and Geroldinger~\cite{GG06} and the references contained therein.

It is often more tricky to prove results about the set of $k$-sums $k\wedge A$ than the $k$-fold sumset $kA$.  For example, it is a straightforward consequence
of the Cauchy-Davenport Theorem~\cite{C,D} that the inequality
\[|kA|\ge \min\{p,k|A|-(k-1)\}\]
holds for each $k\in \N$ and subset $A$ of $\ZZ_p$, the integers modulo a prime $p$.  
Establishing the analogous result for $k$-sums is significantly more difficult.
Indeed, it was a major breakthrough when Dias da Silva and Hamidoune~\cite{DdSH} proved that the inequality 
\begin{equation}
|k\wedge A|\ge \min\{p,k(|A|-k)+1\} \label{THM:DdSH}
\end{equation}
holds for each $k\in \N$ and subset $A$ of $\ZZ_p$ ($p$ prime).
This result was proved by studying a general problem on Grassmann spaces, and the case $k=2$ confirmed a conjecture of Erd\H os and Heilbronn~\cite{EH}.
An alternative proof, using the Combinatorial Nullstellensatz~\cite{Alon99}, was then given by Alon, Nathanson and Ruzsa~\cite{ANR}.\vspace{0.2cm}

Our aim in this paper is to prove a much more modest estimate but in the more general setting of an \emph{arbitrary} abelian group $G$.  In the case of $k$-fold sums it is trivial that $|kA|\ge |A|$, and easily proved that equality occurs if and only if $A$ is a coset of a subgroup of $G$ (provided $k\ge 2$).  We prove the analogous result for $k$-sums.  

We require the following definitions.  An \emph{elementary $2$-subgroup} is a finite subgroup of $G$ in which all non-identity elements have order $2$ (equivalently, a subgroup isomorphic to $\ZZ_2^d$, for some $d\in \N$).  We call an additive set $A\ssq G$ a \emph{$2$-coset} if it is a coset of an elementary $2$-subgroup and an \emph{almost $2$-coset} if $A$ may be obtained from a $2$-coset by removing a single element.\vspace{0.1cm}

\begin{theorem}\label{main} Let $A\ssq G$ be an additive set and let $2\le k\le |A|-2$.  Then 
\[ |k\wedge A|\ge |A|\]
unless $k\in \{2,|A|-2\}$ and $A$ is a $2$-coset.  Furthermore \[ |k\wedge A|>|A|\]
unless $A$ is a coset of a subgroup of $G$ or $k\in \{2, |A|-2\}$ and 
\begin{itemize}
\item[(i)] $A$ is an almost $2$-coset, or
\item[(ii)] $|A|=4$ and $A$ is the union of two cosets of a subgroup of order $2$.
\end{itemize} 
\end{theorem}\vspace{0.1cm}

\begin{remark} The first assertion of Theorem~\ref{main} was first proved by Mann and Olson~\cite{MannOlson67} in the case that $G=\ZZ_p$ ($p$ prime), and by Diderrich~\cite{Diderrich73} in the case that $G$ has no $2$-torsion.  Wang~\cite{Wang08} rediscovered the result of Diderrich and proved that equality occurs if and only if $A$ is a coset of a subgroup of $G$.\end{remark}

\begin{remark} Diderrich's proof is inductive and for this reason he actually proved a slightly more general result stating that $|\sigma(k\wedge A)|\ge |\sigma(A)|$ for any canonical homomorphism $\sigma:G\to G/H$.  Alternatively, this result may be viewed as a sequence version of the inequality $|k\wedge A|\ge |A|$.  In Section~\ref{sec:seqversion}, we sketch the proof of a sequence version of Theorem~\ref{main} which extends Diderrich's result. \end{remark}

Our proof of Theorem~\ref{main} proceeds in two steps.  In the first, we focus on $2$-sums and prove detailed results concerning this case.  In the second, we relate the $k$-sums of a set $A$ to the $2$-sums of a subset $B\ssq A$ using the observation that
\begin{equation} a_1+ \dots + a_{k-2} +\, \, 2\wedge B \, \, \ssq \,\, k\wedge A\, \label{ttk}\end{equation}
for any distinct elements $a_1,\dots ,a_{k-2}$ of $A\setminus B$.  Thus, our proof relies on proving bounds on $|2\wedge B|$ for subsets $B\ssq A$.\vspace{0.1cm}

\begin{proposition}\label{prop:ttk} Let $A\ssq G$ be an additive set of cardinality $|A|\ge 5$.
\begin{itemize}
\item[(i)] If $|2\wedge A|\ge |A|$ then there is a subset $B\ssq A$ with $|B|=\lfloor (|A|+3)  /2\rfloor$ such that $|2\wedge B|\ge |A|$.
\item[(ii)] If $|2\wedge A|<|A|$ then there is a subset $B\ssq A$ with $|B|= |A|/2+1$ such that $|2\wedge B|=|A|-1$.
\end{itemize}
\end{proposition}

The reader will observe that by~\eqref{ttk}, together with Proposition~\ref{prop:ttk}, the implication
\[ |2\wedge A|\ge |A| \qquad \Rightarrow \qquad |k\wedge A|\ge |A|\]
holds for all $k\in \{2,\dots, \lfloor |A|/2\rfloor\}$.  Taking into account the symmetry $|k\wedge A|=|(|A|-k)\wedge A|$, we observe that the implication holds for all $k\in \{2,\dots ,|A|-2\}$.  Except for the special case when $A$ is a $2$-coset, this is precisely our proof of the first part of Theorem~\ref{main} (see Proposition~\ref{first} and its proof).  The proof of the second part of Theorem~\ref{main}, the characterization of the cases where the equality $|k\wedge A|=|A|$ holds, again uses Proposition~\ref{prop:ttk}, but in this case we consider not only a single $(k-2)$-sum of $A\setminus B$ (as in~\eqref{ttk}) but all of them:
\[(k-2)\! \wedge \!(A\setminus B)\,\,\, +\, \,\, 2\wedge B\,\,\, \ssq \,\,\, k\wedge A \, .\label{ttka}\vspace{0.2cm} \]

The outline of the paper is as follows.  We complete the introduction by stating the graph-theoretic analogue of Proposition~\ref{prop:ttk}.  In Section~\ref{sec:two} we prove results concerning $|2\wedge A|$, the number of $2$-sums of $A$.  In Section~\ref{sec:ttk} we prove Proposition~\ref{prop:ttk}, and in Section~\ref{sec:mainproof} we bring together the two strands to prove Theorem~\ref{main}. In Section \ref{sec:seqversion}, we state and prove a sequence version of Theorem~\ref{main}. Finally, in Section~\ref{ConcRem}, we mention some open problems and make some concluding remarks.

We remark that Proposition~\ref{prop:ttk} will be deduced from the following graph-theoretic statement which we state here in case it may be of independent interest.  We denote by $K_n$ the complete graph on $n$ vertices.  The vertex set and edge set of $K_n$ are denoted by $V(K_n)$ and $E(K_n)$ respectively.  Given a graph $G$ and a vertex subset $U$ we denote by $G[U]$ the subgraph of $G$ induced by $U$.  Finally, we recall that a proper edge-colouring of a graph is one in which any two edges that share an endpoint receive different colours.  \vspace{0.1cm}

\begin{proposition}\label{graphprop} Let $n\ge 5$ be an integer and let $c:E(K_n)\to \N$ be a proper edge-colouring of the complete graph on $n$ vertices. 
\begin{itemize}
\item[(i)] If at least $n$ colours are used by $c$ then there is a subset $U\ssq V(K_n)$ with $|U|=\lfloor (n+3)  /2\rfloor$ and such that $c$ uses at least $n$ colours on the edges of $K_n[U]$.
\item[(ii)] If $c$ uses $n-1$ colours then there is a subset $U\ssq V(K_n)$ with $|U|= n/2+1$ and such that $c$ uses $n-1$ colours on the edges of $K_n[U]$.
\end{itemize}
\end{proposition}

The proof of Proposition~\ref{graphprop}, given in Section~\ref{sec:ttk}, is a straightforward application of the probabilistic method, namely the first moment method.

\section{The number of $2$-sums}\label{sec:two}

In this section we prove some lemmas on the number of $2$-sums of an additive set $A$.  Before doing so, we recall the following well known fact (see~\cite[Proposition $2.2$]{TV} for example).\vspace{0.1cm}

\begin{lemma}\label{subgroup} Let $A,B\ssq G$ be additive sets.  Then $|A+B|=|A|$ if and only if $A$ is a union of cosets of a subgroup $H$ of $G$ and $B$ is contained in a coset of $H$.  In particular, $|A+A|=|A|$ if and only if $A$ is a coset of a subgroup of $G$. \end{lemma}

We now establish the key lemmas on the set of $2$-sums of $A$. To do so, the following notation will be used.  
We say that a $2$-sum $g$ of an additive set $A$ is \emph{represented} at an element $a\in A$ if there is an element $b\in A\setminus \{a\}$ such that $a+b=g$.\vspace{0.1cm}

\begin{lemma}\label{lem:two} Let $A\ssq G$ be an additive set of cardinality at least $3$.  Then $|2\wedge A|<|A|$ if and only if $A$ is a $2$-coset.\end{lemma}

\begin{proof}  In the case that $A$ is a $2$-coset then it is easily observed that $|2\wedge A|=|A|-1<|A|$.\vspace{0.2cm}

Now, suppose that $A$ is an additive set for which $|2\wedge A|<|A|$.  Since $|2\wedge A|$ is invariant under translation we may assume also that $0\in A$.  We shall prove that $2y=0$ for every $y\in A$.  Observe that this is sufficient to complete the proof of the lemma.  Indeed, by Lemma~\ref{subgroup} it suffices to prove that $A$ is contained in an elementary $2$-subgroup of $G$ and that $|A+A|=|A|$, and both these facts follow easily once we prove that $2y=0$ for every $y\in A$ (for the latter, observe that $A+A=(2\wedge A) \, \cup\, \{2y:y\in A\}$).

In proving that $2y=0$ for every $y\in A$ we shall use the following observation: since $|A|-1$ distinct $2$-sums are represented at each element of $A$, the inequality $|2\wedge A|<|A|$ may occur only if $|2\wedge A|=|A|-1$.  In this case every $2$-sum is represented at every element of $A$.  

Now, let $y\in A$ be fixed, we prove that $2y=0$.  We argue as follows. Let $z$ be an arbitrary element of $A\setminus \{0,y\}$. Then we may deduce from the above observation that the sum $y+z$ is represented at $0$ and the sum $0+z$ is represented at $y$.  It follows that $$y+z,\, z-y\, \in\, A\, .$$  If these elements are distinct, then $2z=(y+z)+(z-y)$ would be a $2$-sum not represented at $z$, a contradiction.  It follows that $y+z=z-y$, and so $2y=0$, as required.\end{proof}

We now characterize subsets of $2$-cosets for which the inequality $|2\wedge A|\le |A|$ holds.  Recall that an almost $2$-coset is an additive set $A$ that may be obtained from a $2$-coset by removing a single element.\vspace{0.1cm}

\begin{lemma}\label{lem:oratmost} Let $A\ssq G$ be a subset of a $2$-coset. Then $|2\wedge A|>|A|$ unless $A$ is either a $2$-coset or an almost $2$-coset.\end{lemma}

\begin{proof}  Without loss of generality we may assume that $0\in A$ so that $A$ is contained in some elementary $2$-subgroup $H$ of $G$.  We may further assume that $H$ is minimal such that $A\ssq H$, that $\{x_1,\dots ,x_d\}$ is a minimal subset of $A$ such that $H=\langle x_1,\dots ,x_d\rangle$, and that $|H\setminus A|\ge 2$ (as there is nothing to prove if $A$ is a $2$-coset or almost $2$-coset).  Let $I\ssq \{1,\dots,d\}$ be minimal such that $\sum_{i \in I} x_i\not \in A$, and let $J\ssq \{1,\dots,d\}$ be minimal such that $J\neq I$ and $\sum_{j \in J} x_j\not \in A$.  By minimality both $\sum_{i \in I} x_i$ and $\sum_{j \in J} x_j$ may be expressed as sums of two distinct elements of $A$.  Thus, 
$$2\wedge A\, \supseteq\, (A\setminus \{0\})\, \cup\, \left\{\sum_{i \in I} x_i, \sum_{j \in J} x_j\right\}\, ,$$ 
and so $|2\wedge A|>|A|$.\end{proof}

We now establish two lemmas that will help us characterize those additive sets $A$ for which $|2\wedge A|=|A|$.\vspace{0.1cm}

\begin{lemma}\label{iff} Let $A\ssq G$ be an additive set of cardinality at least $5$, and suppose that $A$ is not an almost $2$-coset.  Then $$|2\wedge A|\le |A|\quad \Leftrightarrow \quad |A+A|=|A|\, .$$\end{lemma}

\begin{proof} Since $2\wedge A$ is a subset of $A+A$, one direction of the implication is trivial.  Also, we remark that if $A$ is a $2$-coset then $|A+A|=|A|$.  In particular, by Lemma~\ref{lem:two}, this deals with the case that $|2\wedge A|<|A|$.

Thus, we may suppose that $A$ is an additive set of cardinality at least $5$, that $A$ is not an almost $2$-coset and that $|2\wedge A|=|A|$.  We shall prove that in this case $A+A=2\wedge A$.  To do so, it suffices to prove that for each element $x\in A$ there exists a pair of distinct elements $y,z\in A$ such that $2x=y+z$.\vspace{0.2cm}

Fix $x\in A$.  We first claim that there are three elements $y_1,y_2,y_3$ of $A$ such that $2 x\neq 2 y_i$ for $i=1,2,3$.  Observe that if this were not the case then $A$ would contain a subset $C$ of cardinality $|A|,|A|-1$ or $|A|-2$ such that $2c=2c'$ for all $c,c'\in C$.  In fact each of these three possibilities contradicts the equality $|2\wedge A|=|A|$.  If $|C|=|A|$ then $A$ is a subset of a $2$-coset and so $|2\wedge A|>|A|$ by Lemma~\ref{lem:oratmost}, a contradiction.  If $|C|=|A|-1$ then there is an element $a\in A$ such that $A=C\cup \{a\}$, where $C$ is a subset of a $2$-coset and $a$ is not an element of this coset.  It follows that $2\wedge C$ and $a+C$ are disjoint subsets of $2\wedge A$, and so $|2\wedge A|\ge 2|C|-1 \ge 2|A|-3>|A|$, a contradiction.  Finally, if $|C|=|A|-2$ then there exist elements $a,a'\in A$ such that $A=C\cup \{a,a'\}$, where $C$ is a subset of a $2$-coset and $a,a'$ do not belong to this coset.  Arguing as above we obtain that $|2\wedge A|\ge 2|C|-1\ge 2|A|-5$, which is a contradiction unless $|A|=5$.  In the case $|A|=5$ we note that the subset $C$ has cardinality $3$, and so $|2\wedge A|\ge 3+3>|A|$, as $|2\wedge C|=3$ for all additive sets of cardinality $3$, again a contradiction.

Having established the claim we may now assume that there are three elements $y_1,y_2,y_3$ of $A$ such that $2 x\neq 2 y_i$ for $i=1,2,3$.  Since at most one $2$-sum is not represented at $x$, one of the sums $y_1+y_2,y_1+y_3$ is represented at $x$, without loss of generality $y_1 + y_2$.  Let $u\in A$ be such that $x+u=y_1+y_2$.  Arguing as before, at least one of the sums $x+y_1,x+y_2$ is represented at $u$, without loss of generality $x+y_2$.  Let $z\in A$ be such that $u+z=x+y_2$, and set $y=y_1$.  We claim that $y+z=2x$.  Indeed $$y+z\, =\, y+y_2+z-y_2\, =\, x+u+z-y_2\, =\, x+x+y_2-y_2\, =\, 2x\, .$$\end{proof}

In the case that $|A|=4$ we have the following.\vspace{0.1cm}

\begin{lemma}\label{four} Let $A\ssq G$ be an additive set of cardinality $4$ for which $|2\wedge A|=|A|$.  Then $A$ is the union of two cosets of a subgroup of order $2$.\end{lemma}

\begin{proof} Since there are six pairs of elements of $A$ and only $4$ distinct sums, two sums are repeated.  It follows that the elements of $A$ may be labelled $x,y,z,u$ in such a way that $$x+y\, =\, z+u\qquad \text{and}\qquad x+z\, =\, y+u\, .$$  Hence $$2x\, =\, (x+y)+(x+z)-y-z\, =\, (z+u)+(y+u)-y-z\, =\, 2u\, ,$$ and similarly $2y=2z$.  Thus, $A$ is the union of two cosets of the subgroup $\{0,(x-u)\}=\{0,(z-y)\}$, as required. \end{proof}

\section{Proof of Proposition~\ref{prop:ttk}}\label{sec:ttk}

We first give a proof of Proposition~\ref{prop:ttk} assuming the graph-theoretic Proposition~\ref{graphprop}.  The rest of the section will then be dedicated to proving Proposition~\ref{graphprop}.

\begin{proof}[Proof of Proposition~\ref{prop:ttk}] Let $A\ssq G$ be an additive set of cardinality $n\ge 5$.  Define an edge-colouring of the complete graph on vertex set $A$ by assigning colour $x+y$ to the edge $\{x,y\}$ (for each pair of elements $x,y\in A$).  Since $x+y\neq x+z$ when $y,z$ are distinct, this edge-colouring is proper.  \vspace{0.2cm}

$(i)$ If $|2\wedge A|\ge |A|$ then at least $n=|A|$ colours are used in this colouring.  By Proposition~\ref{graphprop} there exists a subset $B \ssq A$ of cardinality $\lfloor (n+3)/2\rfloor$ such that $c$ uses at least $n$ colours on the edges of $K_n[B]$.  Translating back to additive language, this says precisely that $|2\wedge B|\ge n=|A|$, as required.\vspace{0.2cm}

$(ii)$ If $|2\wedge A|<|A|$ then $n-1=|A|-1$ colours are used in this colouring.  By Proposition~\ref{graphprop} there exists a subset $B \ssq A$ of cardinality $n/2 +1$ such that $c$ uses $n-1$ colours on the edges of $K_n[B]$.  Translating back to additive language, this says precisely that $|2\wedge B|= |A|-1$, as required.\end{proof}

Our approach to proving Proposition~\ref{graphprop} is to select the vertex subset $U$ uniformly at random from the collection of subsets of the appropriate cardinality and prove a lower bound on the expected number of colours used in $K_n[U]$.  The following lemma will help us prove the required bound on the expectation.\vspace{0.1cm}

Given two positive integers $a,n$ such that $n\ge 5$ and $a\le n/2$ we define the probability $p(a,n)\in [0,1]$ as follows. Let $p(a,n)$ be the probability that a subset $U \subseteq \{1,\dots ,n\}$ selected uniformly at random from the collection of subsets of cardinality $\lfloor (n+3)/2\rfloor$ covers at least one of the pairs $\{1,2\},\dots ,\{2a-1,2a\}$.  Note that the same probability $p(a,n)$ applies if the pairs $\{1,2\},\dots ,\{2a-1,2a\}$ are replaced by any $a$ disjoint pairs $e_1,\dots ,e_a$ of elements of $\{1,\dots ,n\}$. 

\begin{lemma}\label{prob} Let $a,n$ be two positive integers such that $n\ge 5$ and $a\le n/2$.  Then $$p(a,n)\ge \frac{2a}{n}\, ,$$ with equality if and only if $a=n/2$.\end{lemma}

\begin{proof}  Let $F$ be a set of $\lfloor n/2\rfloor$ disjoint pairs of elements of $\{1,\dots ,n\}$. 
For each subset $E \subseteq F$ of cardinality $|E|=a$, let $X_E$ be the indicator random variable for $U$ covering at least one pair of $E$.
On the one hand, the pigeonhole principle implies that every subset $U \subseteq \{1,\dots ,n\}$ of cardinality $\lfloor (n+3)/2\rfloor$ covers at least one pair of $F$. 
Thus, we obtain
$$\displaystyle\sum_{\substack{E \subseteq F \\ |E|=a}} X_E \ge {\lfloor n/2\rfloor-1 \choose a-1} \, .$$
On the other hand, it follows from the above discussion that for every $E \subseteq F$ of cardinality $|E|=a$, one has $p(a,n)=\Pr(X_E=1)$. 
Therefore, linearity of expectation yields
$$p(a,n) \ge {\lfloor n/2\rfloor-1 \choose a-1} / {\lfloor n/2\rfloor \choose a} = \frac{a}{\lfloor n/2\rfloor} \, \ge \frac{2a}{n} \, .$$

Now, for equality to hold, $n$ must be even. In particular, $n \ge 6$ and $\lfloor (n+3)/2\rfloor=n/2+1$. 
Moreover, there is a subset $U_0 \subseteq \{1,\dots ,n\}$ of cardinality $n/2+1$ covering at least two pairs of $F$. 
Hence the number of subsets $E \subseteq F$ of cardinality $|E|=a$ such that $U_0$ covers at least one pair of $E$ strictly exceeds ${n/2-1 \choose a-1}$ whenever $a < n/2$.  
Therefore, the equality $p(a,n)=2a/n$ can only occur when $a=n/2$, and it is easily seen that $p(a,n)=1$ holds in this case.
\end{proof}

We now complete our proof of Proposition~\ref{prop:ttk} by proving Proposition~\ref{graphprop}.

\begin{proof}[Proof of Proposition~\ref{graphprop}] $(i)$ Let $U$ be selected uniformly at random from the collection of vertex subsets of cardinality $\lfloor (n+3)/2\rfloor$.  For each $i\in \N$, let $a_i$ denote the number of times colour $i$ is used in the colouring $c$.  Now, by Lemma~\ref{prob}, the expected number of colours used in $K_n[U]$ is
\[\sum_{i\in \N}p(a_i,n)\, \ge\, \sum_{i\in \N}\frac{2a_i}{n}\ =\, n-1\, .\]
Furthermore, since at least $n$ colours are used in the colouring $c$, some $a_i$ must lie strictly between $0$ and $n/2$. 
Thus, the above inequality is strict, and the expected number of colours used in $K_n[U]$ is at least $n$.  
In particular, there exists a subset $U$ of $\lfloor (n+3)/2\rfloor$ vertices such that at least $n$ colours are used in $K_n[U]$.
\vspace{0.2cm}

$(ii)$ Since each colour may be used at most $n/2$ times and $c$ uses $n-1$ colours, each colour class is a perfect matching.  
It is then easily observed that for \emph{every} subset $U$ of $n/2+1$ vertices, 
$K_n[U]$ contains an edge from each of the $n-1$ colour classes.\end{proof}

\section{Proof of Theorem~\ref{main}}\label{sec:mainproof}

In this section we bring together the results of the previous sections to prove Theorem~\ref{main}.  We note that the results of Section~\ref{sec:two} together establish the $k=2$ case of Theorem~\ref{main}.  We may immediately complete the proof of the first part of Theorem~\ref{main}.\vspace{0.1cm}

\begin{proposition}\label{first} Let $A \ssq G$ be an additive set and let $2\le k\le |A|-2$.  Then $$|k\wedge A|\, \ge \, |A|\vspace{0.1cm}$$ unless $k\in \{2,|A|-2\}$ and $A$ is a $2$-coset.\end{proposition}

\begin{proof} By the symmetry $|(|A|-k)\wedge A|=|k\wedge A|$ it suffices to prove the proposition for $2\le k\le \lfloor |A|/2\rfloor$.  The $k=2$ case is precisely Lemma~\ref{lem:two}.  
For $k\in \{3,\dots ,\lfloor |A|/2\rfloor\}$, we may assume $|A|\ge 5$. We now consider the following two cases. 
\vspace{0.2cm}

\textbf{Case I.} If $|2\wedge A|\ge |A|$, then, by Proposition~\ref{prop:ttk}$(i)$, there exists a subset $B\ssq A$ of cardinality $\lfloor (|A|+3)/2\rfloor$ such that $|2\wedge B|\ge |A|$.  Now simply choose a sequence $c_1,\dots ,c_{k-2}$ of $k-2\le |A|/2 -2$ distinct elements of $A\setminus B$ and observe that 
$$k\wedge A\, \supseteq\, (c_1+\dots +c_{k-2})\, +\,\, 2\wedge B\, ,$$ and so $|k\wedge A|\ge |2\wedge B|\ge |A|$, as required.\vspace{0.2cm}

\textbf{Case II.} If $|2\wedge A|<|A|$, then, by Lemma~\ref{lem:two}, $A$ is a $2$-coset, and one easily observes that $2\wedge A$ is an almost $2$-coset.  Also, by Proposition~\ref{prop:ttk}$(ii)$, there exists a subset $B\ssq A$ of cardinality $|A|/2+1$ such that $|2\wedge B|=|A|-1$.  Of course $2\wedge B$ must be precisely $2\wedge A$, and so is also an almost $2$-coset.  Now, let $c_1,\dots ,c_{k-1}$ be distinct elements of $A\setminus B$ and note that the sets 
$$c_1+\dots +c_{k-2}\, +\,\, 2\wedge B\qquad \text{and} \qquad c_2+\dots +c_{k-1}\, +\, \, 2\wedge B$$ are not identical, since $c_{k-1}-c_1$ does not belong to the stabilizer of $2\wedge B$, which is trivial.
Finally, since $k\wedge A$ contains both these sets we deduce that $|k\wedge A|\ge |A|$.\end{proof}

For convenience we read out an immediate corollary of the above proposition, corresponding to the case that $A$ is a coset of a subgroup of $G$, that will be useful in our proof of Theorem~\ref{main}.\vspace{0.2cm}

\begin{corollary}\label{coset} Let $Q$ be a coset of a subgroup $H$ of an abelian group $G$, and let $k\in \{1,\dots ,|Q|-1\}$.  Then $k\wedge Q$ contains exactly those elements of the coset $kQ$, unless $k\in \{2,|Q|-2\}$ and $Q$ is a $2$-coset, in which case $k\wedge Q$ contains all but one element of $kQ$.\end{corollary}

\begin{proof} It is elementary that $k\wedge Q$ is contained in the coset $kQ$ in all cases.  Since $kQ$ is also a coset of $H$ its cardinality is equal to that of $Q$.  The corollary now follows immediately from Proposition~\ref{first} together with the observation that $|2\wedge Q|=|Q|-1$ whenever $Q$ is a $2$-coset.\end{proof}

We now prove a lemma that will help us characterize the cases for which the equality $|k\wedge A|=|A|$ holds.  Given an additive set $A\ssq G$, we write $C(A)$ for the set of sizes of cosets contained in $A$, i.e., $h\in C(A)$ if and only if there is a subgroup $H$ of $G$ of order $h$ such that $Q\ssq A$ for some coset $Q$ of $H$.  Note that $|A|\in C(A)$ if and only if $A$ itself is a coset of a subgroup of $G$.\vspace{0.1cm}

\begin{lemma}\label{lem:aot} Let $A\ssq G$ be an additive set of even cardinality at least $6$.  If $|A|/2\in C(A)$ but $|A|\not \in C(A)$,
then $|k\wedge A|>|A|$ for all $k\in \{3,\dots,|A|-3\}$.\end{lemma}

We must prove a certain \emph{expansion} in the set of sums relative to the original set $A$.  In fact in almost all cases it is easy to prove (using Corollary~\ref{coset}) that at least three cosets of $H$ lie in $k\wedge A$ (where $H$ is a subgroup with $|H|=|A|/2$ and such that a coset of $H$ is contained in $A$) implying $|k\wedge A|\ge 3|A|/2$.  To deal with the few exceptional cases of Corollary~\ref{coset} we use the following observation.\vspace{0.1cm}

\begin{obs}\label{observe} Let $H$ be a subgroup of $G$, let $Q,R$ be cosets of $H$ and let $X\ssq Q, \, Y\ssq R$ be subsets such that $|X|+|Y|>|H|$.  Then $X+Y=Q+R.$\end{obs}  

This observation is simply a variant of the ''prehistoric lemma'' which states that $X+Y=G$ whenever $X,Y\ssq G$ are such that $|X|+|Y|>|G|$.  We now prove Lemma~\ref{lem:aot}.

\begin{proof}[Proof of Lemma~\ref{lem:aot}] Let $H$ be a subgroup of cardinality $|A|/2$ such that some coset $Q$ of $H$ is contained in $A$.  By the usual symmetry argument, it suffices to prove the result for $k\in \{3,\dots ,\lfloor |A|/2\rfloor\}$.  We consider the following three cases.\vspace{0.2cm}

\textbf{Case I.} $A$ is the union of two cosets $Q_1,Q_2$ of $H$ and $k=3$.  Observe first that $2Q_1\neq 2Q_2$ in $G/H$, as $A$ itself is not a coset of a subgroup of $G$.  It follows that the cosets $3Q_1,2Q_1+Q_2$ and $Q_1+2Q_2$ are distinct, and therefore the sets 
\begin{equation}\label{threesubsets} 3\wedge Q_1,\, \, 2\wedge Q_1\,  +\, Q_2 \qquad \text{and}\qquad Q_1\, +\, \, 2\wedge Q_2\end{equation}
are disjoint subsets of $3\wedge A$.  If $|H|=3$ then it is easily observed that these sets have cardinalities $1,3,3$ respectively, which implies that $|3\wedge A|\ge 7>6=|A|$.  If $|H|\ge 4$ then it follows from Corollary~\ref{coset} that each of the sets $j \wedge Q_1, j\wedge Q_2$, where $j\in \{1,2,3\}$, has cardinality at least $|H|-1>|H|/2$, and so, using Observation~\ref{observe}, we obtain that the sets mentioned in~\eqref{threesubsets} have cardinalities at least $|H|-1,|H|,|H|$ respectively, and so $|3\wedge A|\ge 3|H|-1>|A|$, as required.\vspace{0.2cm}

\textbf{Case II.} $A$ is the union of two cosets $Q_1,Q_2$ of $H$ and $k\ge 4$.  Since $k\ge 4$ we may assume $|A|\ge 8$, and so $|H|\ge 4$.  It follows from Corollary~\ref{coset} that each of the sets \[ (k-j)\wedge Q_1\, ,\, \, j\wedge Q_2\qquad\qquad j\in \{1,2,3\}\]
has cardinality at least $|H|-1>|H|/2$.  Observation~\ref{observe} then yields
\[|(k-j)\wedge Q_1 \,+\, \, j\wedge Q_2|\, =\, |H|\qquad \qquad \text{for each }j\in \{1,2,3\}\, .\]
These three sets are disjoint (since, as argued in Case I, $2Q_1\neq 2Q_2$ in $G/H$) and are contained in $k\wedge A$.  It follows that $|k\wedge A|\ge 3|H|>|A|$, as required.\vspace{0.2cm}

\textbf{Case III.} $A$ contains a coset $Q_1$ of $H$ and meets at least two other cosets of $H$, say $Q_2$ and $Q_3$.  By translation we may assume that $Q_1=H$.  Let $x\in A\cap Q_2$ and $y\in A\cap Q_3$.  The sets
\[ x +(k-1)\wedge Q_1\, ,\,\, y +(k-1)\wedge Q_1 \qquad \text{and}\qquad x+y+(k-2)\wedge Q_1\]
are disjoint subsets of $k\wedge A$.  By Corollary~\ref{coset}, each of the three sets mentioned above has cardinality at least $|H|-1$.  Thus, $|k\wedge A|\ge 3|H|-3$.  If $|A|>6$, then $3|H|-3>|A|$ and the proof is complete.  If $|A|=6$ then $|H|=3$ is not a power of $2$, so that Corollary~\ref{coset} gives the stronger conclusion that each of the three sets mentioned above has cardinality at least $|H|$, so that $|k\wedge A|\ge 3|H|=3|A|/2>|A|$, and the proof is complete.\end{proof}

We are now ready to characterize the cases in which the equality $|k\wedge A| =|A|$ holds, completing the proof of our main theorem.

\begin{proof}[Proof of Theorem~\ref{main}] Let $A\ssq G$ be an additive set and let $k\in \{2,\dots ,|A|-2\}$.  
By Proposition~\ref{first}, the inequality \[|k\wedge A|\ge |A|\]
holds, unless $k\in \{2,|A|-2\}$ and $A$ is a $2$-coset.

Now suppose that the equality $|k\wedge A|=|A|$ holds for some $k\in \{2,\dots ,|A|-2\}$.  Then in particular the equality holds for some $k\in \{2,\dots ,\lfloor |A|/2\rfloor\}$.  If $k=2$, then either $|A|=4$, in which case $A$ is a union of two cosets of a subgroup of order $2$ by Lemma~\ref{four}, or $|A|\ge 5$, in which case $A$ is either an almost $2$-coset or a coset of a subgroup of $G$ by Lemmas~\ref{subgroup} and \ref{iff}.

Suppose now that the equality $|k\wedge A|=|A|$ holds for some $k\in \{3,\dots ,\lfloor |A|/2\rfloor \}$.  We prove that $A$ is a coset of a subgroup of $G$.  If $A$ is a $2$-coset then we are done immediately.  Hence we may assume that $A$ is not a $2$-coset.  In particular this implies (by Lemma~\ref{lem:two}) that $|2\wedge A|\ge |A|$ and (by Proposition~\ref{prop:ttk}) that $|2\wedge B|\ge |A|$ for some subset $B\ssq A$ with $|B|=\lfloor (|A|+3)/2\rfloor$.\vspace{0.1cm}

Set \[ S=2\wedge B \qquad \text{and} \qquad C=A\setminus B\, .\vspace{0.3cm}\]

\textbf{Claim.} There is a subgroup $H$ of $G$ such that $S$ is a union of cosets of $H$ and $C$ is contained in a coset of $H$.\vspace{0.3cm}

\textit{Proof of Claim.} By Lemma~\ref{subgroup} it suffices to prove that one of the equalities $|S+C|=|S|$ or $|S-C|=|S|$ holds.  If $k=3$, then $S+C$ is a subset of $k\wedge A$, and so \[|S+C|\le |k\wedge A| = |A| =|S|\, .\]
If $k\in \{4,\dots ,\lfloor |A|/2\rfloor\}$, we shall prove that every subset $C'$ of $C$ of cardinality $k-1$ is contained in a coset of $H$. Since the subset $C'$ is arbitrary, this certainly implies that $C$ is contained in a coset of $H$.  Given a subset $C'$ of $C$ of cardinality $k-1$, we denote by $\sigma_{C'}$ the sum $\sum_{a\in C'}a$.  Now, since $k\wedge A$ contains $S+(k-2)\wedge C'=S+\sigma_{C'}-C'$ we have that \[ |S+\sigma_{C'}-C'|\le |k\wedge A| = |A| = |S| \, .\]
Thus $|S-C'|=|S|$, completing the proof of the claim.\vspace{0.4cm}

It follows from the claim that $|H|$ divides $|S|=|A|$, and $|H|\ge |C|=\lceil (|A|-3)/2\rceil$.  There are three cases to consider.\vspace{0.2cm}

\textbf{Case I.} $|H|=|A|$. In this case $S$ is equal to a coset of $H$, implying that $B$ is contained in a coset $Q$ of $H$.  We claim that $C$ is also contained in $Q$, which implies that $A=B\cup C$ is indeed a coset of a subgroup of $G$.  Indeed, if $C$ were not contained in $Q$ then $k\wedge A$ would contain the $|A|$ elements in $S+(k-2)\wedge C$ and at least $|B|=\lfloor (|A|+3)/2\rfloor$ elements in $B+(k-1)\wedge C$, implying $|k\wedge A|>|A|$, a contradiction.\vspace{0.2cm}

\textbf{Case II.} $|H|=|A|/2$. For this case to occur $A$ must have even cardinality at least $6$, and $S$ is the union of two cosets of $H$.  Note that if $B$ met at least three cosets of $H$ then at least three cosets of $H$ would be represented in $S$.  Thus, $B$ meets exactly two cosets $Q_1,Q_2$ of $H$.  Likewise, if $B$ has at least two elements in each of $Q_1$ and $Q_2$ then at least three cosets of $H$ would be represented in $S$ (the equality $2Q_1=2Q_2$ may not hold, as in this case $B$, and therefore $S$, would be contained in a coset of a subgroup of $G$, and so $|H|=|S|=|A|$, which brings us back to Case I).  The only remaining case is that $B$ contains a coset $Q_1$ of $H$ and a single element of $Q_2$.  In particular $|A|/2=|H|\in C(A)$.  If $|A|\in C(A)$ then $A$ is a coset of a subgroup of $G$, as required.  If $|A|\not\in C(A)$ then it follows from Lemma~\ref{lem:aot} that $|k\wedge A|>|A|$, a contradiction.\vspace{0.2cm}

\textbf{Case III.} $|H|=|A|/3$. In this case $|A|/3=|H|\ge \lceil(|A|-3)/2\rceil$, which implies that $|A|\le 9$ and that $|A|$ is a multiple of $3$, and so $|A|\in \{6,9\}$.  If $|A|=9$, then we have that $|B|=6$, $|S|=9$, $|C|=|H|=3$ and $k\in \{3,4\}$.  We consider first the case $k=3$, so that $|3\wedge A|=|A|$.  By considering the sizes of $C$ and $S$ we see that $C$ is equal to a coset $Q_0$ of $H$ and $S$ is a union of three cosets of $H$.  If $B$ meets at least three cosets of $H$, say elements $b_1,b_2,b_3\in B$ belong to distinct cosets of $H$.  Then \[ 2\wedge Q_0 \,+b_1\, ,\,\, 2\wedge Q_0 \, +b_2\, ,\,\, 2\wedge Q_0 \, +b_3\qquad \text{and}\qquad 3\wedge Q_0\] are disjoint subsets of $3\wedge A$ with cardinalities $3,3,3,1$ respectively, so that $|3\wedge A|\ge 10>9=|A|$, a contradiction.  On the other hand, if $B$ meets only two cosets $Q_1,Q_2$ of $H$, then $B=Q_1\cup Q_2$, and $A=Q_0\cup Q_1\cup Q_2$.  If $Q_0\cup Q_1\cup Q_2$ is a coset of a subgroup of $G$ then we are done, so we may suppose (by relabelling if necessary) that $2Q_1\neq Q_0+Q_2$.  In this case \[Q_0+Q_1+Q_2 \, ,\,\,  Q_0\, +\, \, 2\wedge Q_1\, ,\,\, 2\wedge Q_1\, +Q_2\qquad \text{and}\qquad 3\wedge Q_1\] are disjoint subsets of $3\wedge A$ with cardinalities $3,3,3,1$, so that $|3\wedge A|\ge 10>9=|A|$, a contradiction.

Since the arguments for the case $|A|=9$ and $k=4$ and the case $|A|=6$ are similar to those given above we do not give them in full here, but they may be easily verified by the reader.\end{proof}

\section{A sequence version of Theorem~\ref{main}}\label{sec:seqversion}

In this section we prove an analogue of Theorem~\ref{main} for sequences (or multisets).
We denote by $S(A)$ the support of the sequence $A$. 
Note that $|A|$ now denotes the number of elements in the sequence $A$.  
The following result is mainly an easy consequence of Theorem~\ref{main}. 
However, a few special cases do need to be considered separately.
In the case that $|A|=|S(A)|$ then $A=S(A)$ and one may refer directly to Theorem~\ref{main}.
\vspace{0.1cm}

Given a sequence $A$, we write $S_j(A)$ for the set of elements appearing at least $j$ times in $A$.\vspace{0.1cm}

\begin{theorem}
\label{mainseq} 
Let $A$ be a sequence over $G$ with $|A|>|S(A)|$ and let $2\le k\le |A|-2$.  Then 
\[ |k\wedge A|\ge |S(A)|\]
with equality if and only if one of the following holds.
\begin{itemize}
\item[(i)] $S(A)$ is a coset of a subgroup of $G$.
\item[(ii)] $|S(A)|=2$ and $|S_2(A)|=1$.
\item[(iii)] $S(A)$ is an arithmetic progression of length three, and $S_2(A)$ contains only the middle term of the progression.
\end{itemize}
\end{theorem}

\begin{proof}
As usual it suffices to prove the result for $k\in \{2,\dots ,\lfloor |A|/2\rfloor\}$.
We first deal with the cases where $|S(A)|$ is small. 
If $|S(A)|=1$ then the inequality trivially holds as an equality and $S(A)$ is a coset of the trivial subgroup.  

If $|S(A)|=2$, say $S(A)=\{x,y\}$, then without loss of generality the sequence $A$ contains a subsequence $A'=x,x,y$. Then, either $k=2$ and we readily obtain $|2 \wedge A| \ge |2\wedge A'|=2$, or $k \ge 3$ and there is a subsequence $a_1,\dots ,a_{k-2}$ of $A$ which is disjoint from $A'$ and yields $|k \wedge A| \ge |2\wedge A'|=2$. If $S(A)$ is not a coset of a subgroup and $|S_2(A)|=2$, then $A$ contains a subsequence $A'=x,x,y,y$ and, arguing as above, $|k\wedge A| \ge |2\wedge A'| = 3$. 
Reciprocally, it can easily be checked that $|k \wedge A|=2$ whenever $|S_2(A)|=1$ or $S(A)$ is a coset of a subgroup.

If $|S(A)|=3$, say $S(A)=\{x,y,z\}$, then $A$ contains a subsequence $A'=x,y,z$. Now, either $k=2$ and we have $|2 \wedge A| \ge |2\wedge A'|=3$, or $k \ge 3$ and there is a subsequence $a_1,\dots ,a_{k-2}$ of $A$ which is disjoint from $A'$ and gives $|k \wedge A| \ge |2\wedge A'| = 3$. In addition, if $S(A)$ is not an arithmetic progression of length three then $A$ contains a subsequence of the form $A'=x,y,y,z$ and, arguing as above, $|k\wedge A|\ge |2\wedge A'| = 4$. If $S(A)$ is not a coset of a subgroup but an arithmetic progression of length three such that an element other than the middle term is represented at least twice in $A$, say $A$ contains $A'=x,x,y,z$, then $|k\wedge A| \ge |2\wedge A'| = 4$. Reciprocally, it can easily be checked that $|k \wedge A|=3$ whenever $S(A)$ is a coset of a subgroup or an arithmetic progression of length three such that $S_2(A)$ contains only the middle term of the progression.

From now on, we assume $|S(A)| \ge 4$. 
When $k=2$, we consider two cases. 
If $S(A)$ is a $2$-coset, an almost $2$-coset or a union of two cosets of a subgroup of order $2$, then it is easily seen that, since $|A|>|S(A)|$, the sequence $A$ has at least one more $2$-sum than $S(A)$, which gives the required result for these cases. 
Otherwise, the desired result follows immediately from Theorem~\ref{main}.

Now suppose $3 \le k \le \lfloor |A|/2 \rfloor$. If $|S(A)|=4$, say $S(A)=\{w,x,y,z\}$, then $A$ contains a subsequence $A'=w,x,y,z$. 
First, one can readily notice that if $S(A)$ is a coset of a subgroup then the desired inequality holds as an equality.  
Therefore, we now assume that $S(A)$ is not a coset of a subgroup and consider the following two cases. 
If $S(A)$ is not the union of two cosets of a subgroup of order $2$ then there is a subsequence $a_1,\dots ,a_{k-2}$ of $A$ which is disjoint from $A'$, and it follows from Theorem~\ref{main} that $|k \wedge A| \ge |2 \wedge A'| \ge 5$. 
If $S(A)$ is the union of two cosets of a subgroup of order $2$, then without loss of generality $A$ contains a subsequence $A'=w,w,x,y,z$. 
Now, either $k=3$ and we readily obtain $|3 \wedge A| \ge |3 \wedge A'| = |2 \wedge A'| \ge 5$, or $k \ge 4$ and there is a subsequence $a_1,\dots ,a_{k-3}$ of $A$ which is disjoint from $A'$. 
Therefore, $k\wedge A$ contains a translate of $3 \wedge A'$, so that $|k \wedge A| \ge |3 \wedge A'| = |2 \wedge A'| \ge 5$. 

We now consider the remaining case where $|S(A)| \ge 5$. 
If $3 \le k \le |S(A)|-3$, then the result follows immediately from Theorem~\ref{main}, by simply observing that $|k \wedge A| \ge |k \wedge S(A)|$. 

If $|S(A)|-2 \le k \le \lfloor |A|/2\rfloor$, let $A'$ be a subsequence of $A$ with $A'=S(A')=S(A)$. Since $|A|-|S(A)| \ge 2k - |S(A)| \ge k-2$, there is a subsequence $a_1,\dots,a_{k-3}$ of $A$ which is disjoint from $A'$.   
Thus, it follows from Theorem~\ref{main} that $|k \wedge A| \ge |3 \wedge A'| \ge |S(A)|$.  Furthermore, if $S(A)$ is not a coset of a subgroup then Theorem~\ref{main} gives a strict inequality.
\end{proof}

\section{Concluding Remarks}\label{ConcRem}

The results of this paper establish that an additive set $A$ has at least $|A|$ $k$-sums unless $k \in \{2,|A|-2\}$ and $A$ is a $2$-coset, and strictly more than $|A|$ $k$-sums except in a few specific cases. Though these results are precise in the cases they deal with, we expect that stronger bounds are true in general.  For example, one might hope to prove an analogue of~\eqref{THM:DdSH} for additive subsets $A$ of general abelian groups.  Indeed, Alon, Nathanson and Ruzsa~\cite{ANR} asked this question for $G=\ZZ_n$, where $n$ is not necessarily prime.\vspace{0.2cm}

In another direction one might ask what cardinality of a subset $A$ of an abelian group $G$ ensures that $A$ covers $G$ with its $k$-sums.  
For example, it is easily observed that if $G$ is of odd order and $|A|\ge (|G|+3)/2$, then $2\wedge A=G$ (see~\cite[Lemma 2.2]{GGHHLP02} or~\cite[Lemma 3.2]{Hami98}). 
One might expect that the condition $|A|\ge (|G|+3)/2$ should be sufficient to guarantee that $k\wedge A =G$ for all $k\in \{2,\dots ,|A|-2\}$, and in any abelian group $G$ of odd order.  However, this is not the case in $G= \ZZ_3^2$, where, taking $A$ to be the union of two cosets of a subgroup of order $3$, we have that $|A|=6=(|G|+3)/2$, but $|3\wedge A|=7<9=|G|$.  Yet, we believe this result should hold in all abelian groups of sufficiently large odd order.\vspace{0.2cm}

In the case that $k>2$ one might even expect that weaker conditions also imply that $k\wedge A=G$.  For example, Lev~\cite{Lev02} proved that if $G$ has sufficiently large odd order and $|A|>2|G|/5$ then $3\wedge A=G$.  Furthermore, the condition $|A|>2|G|/5$ may be weakened to $|A|>5|G|/13$ provided that $A$ is not a union of two cosets of a subgroup of order $5$.  This suggests the following general problem.\vspace{0.1cm}

\begin{problem} Determine $c_k:=\limsup_{m\to \infty} c_k(2m+1)$ for each $k\in \N$, where
\[ c_k(n):=\inf\left\{c: A\ssq G, |G|=n, |A|>c|G|\, \Rightarrow k\wedge A=G\right\}\, .\]\end{problem}

\section*{Acknowledgements}
The present paper is the first part of a joint work with Yahya ould Hamidoune. Our research project was ongoing when he passed away in March 2011.
He will keep on being a living source of inspiration to us, and we would like to express all our gratitude for his constant mathematical enthusiasm.

\end{document}